\RequirePackage{fix-cm}

\documentclass[smallextended]{svjour3}       
\smartqed  
\usepackage{graphicx}
\usepackage{amssymb}
\usepackage{amsmath}
\usepackage{stmaryrd}
\usepackage{graphicx}
\usepackage{placeins}
\usepackage{pgfplots}
\usepackage{tikz}
\usetikzlibrary{shapes,arrows,calc,external}

\newtheorem{assumption}[theorem]{Assumption}

\DeclareMathOperator{\R}{\mathbb{R}}

\newcommand{\dx}{\, \mathrm{d} x}

\newcommand{\dt}{\Delta t}

\DeclareMathOperator{\cfl}{CFL}

\def\Vs{{V_{\gamma}}}
\def\l2norm#1{\|#1\|_{L^2}} 
\def\h1norm#1{\|#1\|_{H_0^1}} 
\def\snorm#1{\|#1\|_{\gamma}} 
\def\sdnorm#1{\|#1\|_{V_{\gamma}'}}

\def\v#1{v^{(#1)}}
\def\u#1{u^{(#1)}}
\def\eps{\varepsilon}

\begin{document}

\title{An asymptotic preserving method for linear systems of balance laws based on Galerkin's method.}


\titlerunning{Asymptotic Preserving}        

\author{Jochen Sch\"utz}

\institute{J. Sch\"utz\at
           Institut f\"ur Geometrie und Praktische Mathematik, RWTH Aachen University \\ Templergraben 55, 52062 Aachen \\
              Tel.: +49 241 80 97677 \\
              \email{schuetz@igpm.rwth-aachen.de}           
}

\date{Received: date / Accepted: date}

\maketitle

\begin{abstract}
We apply the concept of Asymptotic Preserving (AP) schemes \cite{Jin99} to the linearized $p-$system and discretize the resulting elliptic equation using standard continuous Finite Elements instead of Finite Differences. The fully discrete method is analyzed with respect to consistency, and we compare it numerically with more traditional methods such as Implicit Euler's method.
\end{abstract}

\section{Introduction}
{Approximating solutions of singularly perturbed partial differential equations is a difficult task, see \cite{KC} for an introduction to such PDE.
Considering the Euler equations and the limit process of Mach number $Ma$ towards zero, it is known that these equations change type \cite{KlMa81}, and therefore constitute a singular limit. 
As a consequence, the $\cfl$-condition for explicit schemes prescribes an extremely small timestep $\Delta t$, with $\Delta t \rightarrow 0$ as $Ma \rightarrow 0$. One potential remedy is to use implicit time discretizations, however, it is known that they are overly diffusive and deteriorate the quality of the solution \cite{Kroener}. To this end, the concept of \emph{asymptotic preserving} (AP) schemes (in the context of compressible flows also called \emph{all speed} schemes) has been introduced, see, e.g., the review paper by Jin \cite{Jin2012} and the references therein. (For an excellent historical overview, we refer to \cite{CoDeKu12}.) In contrast to standard schemes, the temporal variable $t$ is discretized first, leaving the spatial variable $x$ continuous. 
Then, an additional equation is derived that is treated implicitly. It is only after this step that $x$ is discretized. In this work, we are interested in extending the concept from Finite-Volume-type to Galerkin-type discretizations. 

We consider the $p-$system \cite{DA} with a linear pressure function $p(v) := -\frac{1}{\varepsilon^2} v$ and a right-hand side $g$,
\begin{alignat}{2}
 \label{eq:p-system1}
 v_t - u_x &= 0 &\quad& \forall (x, t) \in \Omega \times \R^+ \\
 \label{eq:p_system2}
 u_t + p(v)_x &= g(x,t) &\quad& \forall (x, t) \in \Omega \times \R^+ 
 \end{alignat}
on a domain $\Omega \subset \R$ subject to suitable initial and boundary values, where for simplicity we choose the latter to be 
\begin{align}
  \label{eq:bdry_cond_v}
  v(x, t) = 0 \quad \forall (x, t) \in \partial \Omega \times \R^+. 
\end{align}
In a (simplified) physical application, $u$ and $v$ could denote velocity and (variations of) the specific volume of the fluid. 

Obviously, the equation can be written as 
\begin{alignat}{2}
 \label{eq:conslaw}
 w_t + f(w)_x &= G(x, t) &\quad& \forall (x, t) \in \Omega \times \R^+
\end{alignat}
for $w := (v, u)^T$, $f(w) := (-u, -\frac1{\varepsilon^2} v)^T$ and $G(x, t) := (0, g(x, t))^T$. 

The eigenvalues of the Jacobian of the flux function $f$ are $\pm\frac{1}{\varepsilon}$, and so a fully explicit Finite-Volume scheme will not be feasible for small values of $\varepsilon$, as the time-step will decrease with $\varepsilon$. 
Inspired by \emph{Asymptotic Preserving Schemes} (AP), we develop a new solver for \eqref{eq:p-system1}-\eqref{eq:p_system2} based on a combination of Finite Volumes and Finite Elements. Its (fully discrete) consistency is investigated, and it is compared with more traditional numerical schemes with respect to error versus mesh size. 
We put this in the simple framework of the $p-$system because it was on a similar system that Jin \cite{Jin99} derived his famous asymptotic preserving schemes for the first time, and because it is simple (but not too simple), so that each step can be easily computed, which is not the case for more involved systems such as Euler's equations. 

As already mentioned, the concept of asymptotic preserving schemes that we pursue in this publication has been introduced by Jin \cite{Jin99}, building on joint work with Pareschi and Toscani \cite{Jin1998}. 
In these publications, a scheme is called asymptotic preserving if
\begin{itemize}
 \item it is for $\varepsilon \rightarrow 0$ a consistent scheme for the multiscale limiting equations of \eqref{eq:p-system1}-\eqref{eq:p_system2} and
 \item is stable with a $\cfl$-number independent of $\varepsilon$. 
\end{itemize}
This class of schemes has since been extended to various kinds of equations, such as, e.g., Euler's equation \cite{ArNoLuMu12,CoDeKu12}, Shallow-water equations \cite{DeTa}, highly an\-iso\-tropic stationary elliptic equations \cite{DegLoNaNe12} and many more.

The current paper is a first attempt to extend the AP schemes for instationary problems, which have mostly been presented for Finite-Volume discretizations, to Galerkin-type schemes. Based on a flux-splitting, we derive an elliptic equation whose diffusion coefficient is dependent on $\varepsilon$ and $\Delta t$. This equation is solved by continuous Finite-Element methods, and not, as usual, by finite-difference schemes. The approach, though it can of course also be written in terms of finite differences, has the advantage that we can investigate the elliptic equation and its discretization in a rigorous setting in the context of Sobolev spaces. In a first step, we show that the elliptic equation is well-posed and uniformly well-conditioned for all values of $\varepsilon$ and $\Delta t$. This is achieved by introducing problem-dependent spaces and norms. In a second step, we restrict ourselves to 'small' $\varepsilon$ and 'large' $\Delta t$, i.e., $0 < \varepsilon \leq \varepsilon_0 < 1$ and $\Delta t \geq \
varepsilon$, as it is only in this 
setting that we can use standard Finite-Element schemes \cite{Ciarlet1978,GrRo} instead of stabilized ones \cite{BroHugh82}. Also for this setting, we can derive rigorous and uniform (in $\varepsilon$) stability and consistency bounds. 

Solutions to \eqref{eq:p-system1}-\eqref{eq:p_system2} that allow for a limit solution as $\varepsilon \rightarrow 0$ have a certain structure (see \eqref{eq:conseq_multiscale}-\eqref{eq:conseq_multiscale2} in Sec. 2). Our consistency analysis for the fully discrete algorithm heavily relies on this  structure, and we believe that it is only in this setting that one can derive suitable bounds on the consistency error that do not behave like $O(\varepsilon^{-1})$ or even worse. As an easy consequence, we can indeed show that the proposed scheme is AP. This is different to other authors \cite{DeTa,ArNo12} who show that their scheme is asymptotic preserving by a Taylor series argument on the semi-discrete stage. 

Having presented our scheme, we compare it numerically with two other schemes. The surprising outcome is that the scheme to be presented performs better by orders of magnitude in comparison to more traditional schemes. 

The outline of the paper is as follows: In Sec. \ref{sec:multiscalelimit}, we derive the multiscale limit solution of the linearized $p-$system for $\varepsilon \rightarrow 0$. In Sec. \ref{sec:fluxsplitting}, we split the conservative flux $f$ into a stiff $\widetilde f$ and a non-stiff $\widehat f$. Based on this splitting, we derive a semi-discretization in Sec. \ref{sec:semidiscretization}. This yields an elliptic equation, which is investigated in Sec. \ref{sec:elliptic_equation}. Finally, in Sec. \ref{sec:full_discr}, we formulate the fully discrete algorithm and investigate its consistency in Sec. \ref{sec:consistency}. In Sec. \ref{sec:numerics}, we show numerical results. Sec. \ref{sec:outlook} offers conclusions and outlook.}

\section{Asymptotic Preserving Discretization}
\subsection{Multiscale limit of the equation}\label{sec:multiscalelimit}
{In this section, we follow a multiscale approach to obtain the limiting equations of \eqref{eq:conslaw}. 
%
To this end, we assume that our unknown solution $(v, u)$ admits a two-scale expansion as
\begin{align}
 \label{eq:multiscale1}
 v &= v^{(0)} + \varepsilon v^{(1)} + \varepsilon^2 v^{(2)} + O(\varepsilon^3) \\
 \label{eq:multiscale2}
  u &= u^{(0)} + \varepsilon u^{(1)} + \varepsilon^2 u^{(2)} +  O(\varepsilon^3). 
\end{align}
Note that this approach does not include fast waves, i.e., contributions depending on $\frac{1}{\varepsilon}$, so one has a uniform limit as $\varepsilon \rightarrow 0$. As we are dealing with smooth solutions in this ansatz, we consider the $C^1(\Omega \times \R^+)$ topology, i.e., we consider the norm 
\begin{align}
 \|\varphi\|_{C^1} &:= \|\varphi\|_{\infty} + \|\nabla_{x, t} \varphi\|_{\infty},
\end{align}
and \eqref{eq:multiscale1}-\eqref{eq:multiscale2} have to be understood in the sense that 
\begin{alignat}{2}
 \label{eq:multiscale1_top}
  \|v &- v^{(0)} - \varepsilon v^{(1)} - \varepsilon^2 v^{(2)}\|_{C^1} &\ =\ & O(\varepsilon^3) \\
 \label{eq:multiscale2_top}
  \|u &- u^{(0)} - \varepsilon u^{(1)} - \varepsilon^2 u^{(2)} \|_{C^1} &\ =\ &  O(\varepsilon^3). 
\end{alignat}
With this rather strict notion of approximation, we can derive the limiting equations: 
Plugging \eqref{eq:multiscale1}-\eqref{eq:multiscale2} into \eqref{eq:p-system1}-\eqref{eq:p_system2} and balancing the powers of $\varepsilon$ yields that both $v^{(0)}(x, t)$ and $v^{(1)}(x, t)$ are independent of $x$. Therefore, $v^{(1)}(x, t)$ can be absorbed into $v^{(0)}(x, t)$, and \eqref{eq:multiscale1} reduces to 
\begin{align}
 v &= v^{(0)}(t) + \varepsilon^2 v^{(2)} + O(\varepsilon^3).
\end{align}
The remaining limiting equations can be easily seen to be
\begin{alignat}{2}
 \label{eq:p-system1limit}
 v^{(0)}_t - u^{(0)}_x &= 0 &\quad& \forall (x, t) \in \Omega \times \R^+ \\
 \label{eq:p_system2limit}
 u^{(0)}_t - v^{(2)}_x &= g(x,t) &\quad& \forall (x, t) \in \Omega \times \R^+ .
 \end{alignat}
A suitable algorithm approximating \eqref{eq:p-system1}-\eqref{eq:p_system2} for small values of $\varepsilon$ should, in the vanishing $\varepsilon-$limit, be a consistent approximation to \eqref{eq:p-system1limit}-\eqref{eq:p_system2limit}. In reference \cite{Jin99}, such a consistency requirement is called \emph{asymptotic preserving}. 

For a general conservation law, it is nontrivial to obtain more precise results concerning $v^{(0)}$ and $u^{(0)}$, see, e.g., \cite{KlMa81} for results in the context of Euler's equations. However, in the very simple setting of the linearized $p-$system, we can clarify even more the relation between $v$ and $u$:

\begin{lemma}
  A pair of smooth functions $(v, u)$ that admits a two-scale expansion as in \eqref{eq:multiscale1}-\eqref{eq:multiscale2} necessarily has the following form:
  \begin{align}
    \label{eq:conseq_multiscale}
    v(x, t) &= \varepsilon^2 v^{(2)}(x, t) + O(\eps^3) \\
    \label{eq:conseq_multiscale2}
    u(x,t)  &= u^{(0)}(t) + \varepsilon u^{(1)}(t)  + \varepsilon^2 u^{(2)}(x, t) + O(\varepsilon^3)
  \end{align}
   for functions $v^{(2)}, u^{(2)} : \R \times \R^+ \rightarrow \R$ and  $u^{(0)}, \u1 : \R^+ \rightarrow \R$.
\begin{proof}
 Plugging the multiscale expansion \eqref{eq:multiscale1}-\eqref{eq:multiscale2} into the conservation law \eqref{eq:p-system1}-\eqref{eq:p_system2}, one obtains 
 \begin{alignat}{3}
  \label{eq:eqms1}
  \v0_t + \eps \v1_t + \eps^2 \v2_t &- \u0_x - \eps \u1_x - \eps^2 \u2_x &=& \ &O(\eps^3)& \\
  \label{eq:eqms2}
  \u0_t + \eps \u1_t + \eps^2 \u2_t &- \eps^{-2} \v0_x - \eps^{-1} \v1_x - \v2_x \ &=& \ g \ + &O(\eps)&. 
 \end{alignat}
 Considering $O(\eps^{-2})$ and $O(\eps^{-1})$ parts of \eqref{eq:eqms2}, one obtains that both $\v0_x = \v1_x = 0$. Together with the boundary conditions \eqref{eq:bdry_cond_v} imposed on $v$, one can conclude that $\v0 = \v1 = 0$. This knowledge inserted into \eqref{eq:eqms1} and considering $O(1)$ and $O(\eps)$ terms, yields $\u0_x = \u1_x = 0$.
\end{proof}
\end{lemma}

\begin{remark}\label{rem:order_vux}
 Note that both $v(x,t)$ and $u_x(x,t)$ are of order $ O(\varepsilon^2)$. We will use this extensively when performing the consistency analysis of our algorithm. 
\end{remark}
}
\subsection{Flux Splitting}\label{sec:fluxsplitting}
{The way of splitting the flux into stiff and non-stiff parts has an influence on the final algorithm. We choose our splitting according to the following definition:
\begin{definition}\label{def:splitting}
 Let the flux function $f$ be split into $f(w) = \widehat f(w) + \widetilde f(w)$. We consider such a splitting to be admissible if for all $0 < \varepsilon < 1$
 \begin{itemize}
  \item both $\widehat f(w)$ and $\widetilde f(w)$ induce a hyperbolic system, i.e., the eigenvalues of both $\widehat f'(w)$ and $\widetilde f'(w)$ are distinct and real,
  \item the eigenvalues of $\widehat f'(w)$ are of order one, 
  \item $\widehat f(w)$ approaches $f(w)$ as $\varepsilon \rightarrow 1$, and
  \item $\widetilde f(w)$ approaches $f(w)$ for $\varepsilon \rightarrow 0$ in the sense that $\lim_{\varepsilon \rightarrow 0} \varepsilon^2 \left(\widetilde f(w) - f(w)\right) = 0.$
 \end{itemize}
 $\widehat f(w)$ is called the 'non-stiff', and $\widetilde f(w)$ the 'stiff' part of the flux function for obvious reasons. 
\end{definition}

To identify stiff and non-stiff parts of the flux function, we make the following ansatz: 
\begin{align}
  f(w) = \widehat f(w) + \widetilde f(w) =: 
		  \left( \begin{array}{c} -\alpha(\varepsilon) u \\ -\frac{\beta(\varepsilon)}{\varepsilon^2} v \end{array} \right)
		+ \left( \begin{array}{c} -(1-\alpha(\varepsilon)) u \\ -\frac{1-\beta(\varepsilon)}{\varepsilon^2} v \end{array} \right).
\end{align}
Both $\alpha(\cdot)$ and $\beta(\cdot)$ are yet unknown. One reasonable requirement is $\alpha(1) = \beta(1) = 1$, and $\alpha(0) = \beta(0) = 0$, so that one has no stiff contribution given that $\varepsilon$ is one, and no non-stiff contribution given that $\varepsilon$ vanishes. We make the simple ansatz of $\alpha(\varepsilon) = \varepsilon^a$, $\beta(\varepsilon) = \varepsilon^b$. An easy computation shows that for $a, b > 0$, $a + b = 2$, the eigenvalues of $\widehat f'(w)$ are independent of $\varepsilon$. A particularly simple choice is $a = b = 1$, which we will use throughout this work. In summary, for this choice of $a$ and $b$, we have
\begin{alignat}{3}
 \widehat f(w) &= \left( \begin{array}{c} -\varepsilon u \\ -\frac{1}{\varepsilon} v \end{array} \right), &\quad& 
 \widetilde f(w) &= \left( \begin{array}{c} -(1-\varepsilon) u \\ -\frac{1-\varepsilon}{\varepsilon^2} v \end{array} \right)
\end{alignat}
with corresponding eigenvalues of the Jacobians
\begin{align}
 \widehat \lambda &= \pm 1, \quad \widetilde \lambda = \pm \frac{1-\varepsilon}{\varepsilon}.
\end{align}
}
\subsection{Semi-Discretization}\label{sec:semidiscretization}
{We start the description of our algorithm with a discretization in time only. For simplicity, we assume that we work on space-time slabs of (uniform) size $\Delta t$, although uniformity is not a necessary condition. 
Throughout this work, we will use standard notation and set $w^n := w(t^n)$, where $t^n :=  n \Delta t$. 
Based on the flux splitting defined in Sec. \ref{sec:fluxsplitting}, we obtain a first-order implicit / explicit semidiscretization of $\eqref{eq:conslaw}$ in time, given by
\begin{alignat}{2}
 \label{eq:semidiscretization}
 \frac{w^{n+1}-w^n}{\dt} + \widehat f(w^n)_x + \widetilde f(w^{n+1})_x = G^n
\end{alignat}
or, in terms of $(v, u)$,
\begin{alignat}{2}
 \label{eq:imex1}
 \frac{v^{n+1}-v^n}{\dt} &= \varepsilon u_x^{n} + (1-\varepsilon)u_x^{n+1} \\
 \label{eq:imex2}
 \frac{u^{n+1}-u^n}{\dt} &= \frac{1}{\varepsilon} v_x^{n} + \frac{1-\varepsilon}{\varepsilon^2} v_x^{n+1} + g^n.
\end{alignat}
One way of dealing with such a system of implicit equations that has become a standard ingredient in asymptotic preserving schemes, is to equivalently reformulate \eqref{eq:imex1}-\eqref{eq:imex2} in such a way that one obtains an equation for either $v^{n+1}$ or $u^{n+1}$ alone. We have decided to formulate an equation for $v^{n+1}$. To this end, we note that \eqref{eq:imex2} is equivalent to
\begin{alignat}{2}
\label{eq:u_update}
 u^{n+1} = u^n + \dt \left(\frac{1}{\varepsilon} v_x^{n} + \frac{1-\varepsilon}{\varepsilon^2} v_x^{n+1} + g^n \right), 
\end{alignat}
and plug this into \eqref{eq:imex1}: 
\begin{alignat}{2}
   v^{n+1} &= v^n + \dt \left( \varepsilon u_x^{n} + (1-\varepsilon) \left( u^n + \dt \left(\frac{1}{\varepsilon} v_x^{n} + \frac{1-\varepsilon}{\varepsilon^2} v_x^{n+1} + g^n \right)  \right)_x \right) \\
 &= v^n + \dt \ u_x^n + \dt^2 (1-\varepsilon) \left(\frac{1}{\varepsilon} v_{xx}^{n} + \frac{1-\varepsilon}{\varepsilon^2} v_{xx}^{n+1} + g_x^n \right) \\
 &= v^n + \dt \ u_x^n + \frac{\dt^2(1-\varepsilon)}{\varepsilon}v_{xx}^{n} +  \frac{\dt^2 (1-\varepsilon)^2}{\varepsilon^2} v_{xx}^{n+1} + \dt^2(1-\varepsilon) g_x^n.
\end{alignat}
Rearranging terms yields an elliptic equation for $v^{n+1}$: 
\begin{alignat}{2}
 \label{eq:elliptic_v}
 -\frac{\dt^2 (1-\varepsilon)^2}{\varepsilon^2} v_{xx}^{n+1} + v^{n+1} &= v^n + \dt \ u_x^n + \frac{\dt^2(1-\varepsilon)}{\varepsilon}v_{xx}^{n} + \dt^2(1-\varepsilon) g_x^n. 
\end{alignat}
\begin{remark}
 \eqref{eq:elliptic_v} is a well-posed equation for $\Delta t > 0$ and $0 < \varepsilon < 1$, as the diffusion coefficient $\gamma := \frac{\dt^2 (1-\varepsilon)^2}{\varepsilon^2}$ is strictly positive. However, $\gamma$ is only bounded away from zero for $\Delta t \gg 0$ and $\varepsilon \ll 1$, so one cannot expect to get uniform stability bounds in the $H^1-$norm. 
 Nevertheless, it is possible to obtain uniform bounds in a $\gamma-$dependent norm, see Sec. \ref{sec:elliptic_equation}. 
\end{remark}

The weak formulation of \eqref{eq:elliptic_v} can be cast in a variational framework as 
\begin{align}
 \label{eq:op_poisson}
 a(v^{n+1}, \varphi) = \iota(\varphi) &\quad \forall \varphi \in H_0^1(\Omega), 
\end{align}
where 
\begin{align}
 \label{eq:bilinearform}
 a(v^{n+1}, \varphi) &:= \int_{\Omega} \left(\frac{\dt^2 (1-\varepsilon)^2}{\varepsilon^2} v^{n+1}_{x}  \varphi_x + v^{n+1} \varphi \right) \dx \quad \text{   and } \\ 
 \label{eq:iota}
 \iota(\varphi) &:= \int_{\Omega} \left( v^n + \dt \ u_x^n\right) \varphi - \dt^2(1-\varepsilon)\left(\frac{v_{x}^{n}}\varepsilon  + g^n\right) \varphi_x \dx.
\end{align}
Boundedness and coercivity properties of $a(\cdot, \cdot)$ will be discussed in the next sections. What concerns $\iota$, we can state the following lemma:
\begin{lemma}
 Let us assume that $u^n \equiv u(t^n)$ and $v^n \equiv v(t^n)$ are functions in $H^1(\Omega)$; $g^n \equiv g(t^n)$ is a function in $L^2(\Omega)$;  and $0 < \varepsilon < 1$. Then $\iota \in H_0^1(\Omega)'$. 
\begin{proof}
 It is enough to show that both $v^n + \dt \ u_x^n$ and $ \dt^2(1-\varepsilon)\left(\frac{v_{x}^{n}}\varepsilon  + g^n\right)$ are functions in $L^2(\Omega)$, which is correct because of the assumptions on $u^n, v^n$ and $g^n$. 
\end{proof}
 \end{lemma}
}
\subsection{A note on the elliptic equation}\label{sec:elliptic_equation}
{Let us now turn to the variational equation \eqref{eq:op_poisson}. To make it a well-defined and a uniformly well-conditioned problem for all $0 < \varepsilon \leq 1$, we put it in a variational framework with weighted Sobolev spaces as follows:
\begin{definition}
 Let the coefficient of the viscous term of \eqref{eq:op_poisson} be denoted by $\gamma$, i.e., 
 \begin{align}
  \gamma := \frac{\dt^2 (1-\varepsilon)^2}{\varepsilon^2}.
 \end{align}
 We define a weighted norm $\snorm\cdot$ as
 \begin{align}
  \snorm{\varphi}^2 := \l2norm{\varphi}^2 + \gamma \l2norm{\varphi_x}^2
 \end{align}
 and a corresponding 'Sobolev-space' 
 \begin{align}
  \Vs(\Omega) := \overline{C_0^{\infty}(\Omega)}^{\snorm{\cdot}}.
\end{align}
\end{definition}
\begin{corollary}
 For $\gamma >0$, i.e., $\varepsilon < 1$, the weighted norm $\snorm{\cdot}$ is equivalent to the standard Sobolev norm, as can be seen from a Poincar\'e-Friedrichs inequality. However, the equivalence constants get worse as $\varepsilon$ approaches one. With this equivalence in mind, it is easy to see that 
 \begin{align}
  \Vs(\Omega) = \begin{cases} H_0^1(\Omega), &\quad \varepsilon < 1 \\ L^2(\Omega), &\quad \varepsilon = 1 \end{cases},
 \end{align}
 as $\gamma = 0$ for $\varepsilon = 1$ and $\gamma> 0$ for $0 < \varepsilon < 1$. 
\end{corollary}
\begin{remark}
 The weighted norm $\snorm{\cdot}$ is the energy norm associated to \eqref{eq:op_poisson}, i.e., 
 \begin{align}
    \label{eq:norm_skalarprodukt}
    \snorm{\varphi}^2 = a(\varphi, \varphi).
 \end{align}
 Furthermore, for $\gamma = 0$, the problem \eqref{eq:op_poisson} is not well-posed in $H_0^1(\Omega)$ any more, so the choice of $\Vs(\Omega)$ is actually very natural.
\end{remark}

The following lemma computes both coercivity and boundedness constants of $a(\cdot, \cdot)$ on $\Vs(\Omega)$:
\begin{lemma}
 The bilinear form $a(\cdot, \cdot)$ as defined in \eqref{eq:bilinearform} is coercive on $\Vs(\Omega)\times\Vs(\Omega)$ with ellipticity constant one, and bounded on $\Vs(\Omega)\times\Vs(\Omega)$ with boundedness constant also one. 
 \begin{proof}
  It is easy to see that
  \begin{align}
   a(\varphi, \varphi) = \gamma \l2norm{\varphi_x}^2 + \l2norm{\varphi}^2 = \snorm{\varphi}^2,
  \end{align}
  so the bilinear form is elliptic with ellipticity constant one. 
  Furthermore, using Cauchy-Schwartz inequality (this is possible because of \eqref{eq:norm_skalarprodukt}), one has
  \begin{align}
   a(\varphi, \psi) \leq \snorm{\varphi} \snorm{\psi}.
  \end{align}
 \end{proof}
\end{lemma}

A problem is called \emph{well-conditioned}, if the relative error in the output is bounded by a constant times the relative error in the input. In the current case, input is two functionals $\iota, \widetilde \iota \in \Vs(\Omega)'$, and output is two corresponding solutions $v,\widetilde v$ to the elliptic equation \eqref{eq:op_poisson}, so well-conditioned means that there is a constant $C \in \R$, such that 
 \begin{align}
  \frac{\snorm{v - \widetilde v}}{\snorm{v}} \leq C \frac{\sdnorm{\iota - \widetilde \iota}}{\sdnorm{\iota}}.
 \end{align}
The following theorem guarantees that \eqref{eq:op_poisson} is, for the full range of $0 < \varepsilon \leq 1$, a well-conditioned problem with $C \equiv 1$:
\begin{theorem}\label{thm:konditionszahl_cont}
 The equation \eqref{eq:op_poisson} is well-conditioned in $\Vs(\Omega)$ independently of $\varepsilon$, i.e., for two functionals $\iota, \widetilde \iota \in \Vs(\Omega)'$, and their corresponding solutions $v$ and $\widetilde v$, one has the relation
 \begin{align}
  \label{eq:konditionszahl_cont}
  \frac{\snorm{v - \widetilde v}}{\snorm{v}} \leq \frac{\sdnorm{\iota - \widetilde \iota}}{\sdnorm{\iota}}.
 \end{align}
\begin{proof}
It is a classical result from the theory of elliptic pde that the quotient of boundedness constant and ellipticity constant is indeed the condition number with respect to a perturbation of the functional $\iota$. Nevertheless, for convenience, we give a sketch of the proof.
From ellipticity, we can conclude
\begin{align}
 \label{eq:elliptic_ineq_cont}
  \snorm{v - \widetilde v}^2 = a(v - \widetilde v, v - \widetilde v) = \iota(v-\widetilde v) - \widetilde \iota (v - \widetilde v) \leq \sdnorm{\iota - \widetilde \iota} \snorm{v - \widetilde v}
\end{align}
and from boundedness
\begin{align}
 \label{eq:bounded_ineq_cont}
 \sdnorm{\iota} = \sup_{u\in \Vs(\Omega), \snorm{u} = 1} \iota(u) = \sup_{u\in \Vs(\Omega), \snorm{u} = 1} a(v, u) \leq \snorm{v}. 
\end{align}
\eqref{eq:elliptic_ineq_cont}-\eqref{eq:bounded_ineq_cont} yields \eqref{eq:konditionszahl_cont}.
\end{proof}
 \end{theorem}}
\subsection{Full discretization}\label{sec:full_discr}
{
In this section, we introduce the fully discrete method. To this end, we assume that our spatial domain $\Omega$ is subdivided into cells $\Omega_i$ as 
\begin{align}
 \Omega = \bigcup_{i=1}^{N_x} \Omega_i :=\bigcup_{i=1}^{N_x} [x_i, x_{i+1}]
\end{align}
with midpoints 
\begin{align}
\overline x_i := \frac{x_{i+1} + x_i}{2}. 
\end{align}
For simplicity, we consider a uniform discretization, i.e., 
\begin{align}
 \Delta x := x_{i+1}-x_i
\end{align}
is assumed to be constant. This, however, is only for the ease of presentation, there is no need to have uniform cells. 

In a Finite-Volume fashion, we define approximations $w_i^n \equiv (v_i^n, u_i^n)$ to the quantities $w(\overline x_i, t^n)$ to be piecewise constants. At $t = 0$, we initialize 
\begin{align}
 w_i^0 := w_0(\overline x_i, 0) \quad \forall \ i = 1, \ldots, N_x
\end{align}
for given initial values $w_0: \R \rightarrow \R$ to the conservation law \eqref{eq:conslaw}. 

The overall algorithm relies on the following steps:
\begin{enumerate}
 \item Compute an approximate solution $v$ to \eqref{eq:op_poisson} with (linear) Finite-Elements. 
 \item Update $u$ motivated by \eqref{eq:u_update}.
\end{enumerate}

Let us discuss these  steps separately: Obviously, the variational equation \eqref{eq:op_poisson} can not be solved exactly, because $\iota(\varphi)$ is not available, and one cannot solve the variational equation exactly either. So one first has to start with the definition of an approximation $\iota_h(\varphi)$ to $\iota(\varphi)$. 
$\iota(\varphi)$ is defined by (see also \eqref{eq:iota})
\begin{align}
  \iota(\varphi) &:= \int_{\Omega} \left( v^n + \dt \ u_x^n\right) \varphi - \dt^2(1-\varepsilon)\left(\frac{v_{x}^{n}}\varepsilon  + g^n\right) \varphi_x \dx  \\ 
                 &=: \int_{\Omega} \iota_1 \varphi - \Delta t^2 (1-\varepsilon) \iota_2 \varphi_x \dx. 
\end{align}
We replace both functions $\iota_1$ and $\iota_2$ by piecewise constant quantities $\iota_{1,h}$ and $\iota_{2,h}$. (Note that piecewise constant functions are still in $L^2(\Omega)$ on a bounded domain $\Omega$.) 
Note furthermore that the only non-trivial term to define is the approximation to both $v_x^n$ and $u_x^n$. We define the approximate derivates $\widetilde w_x^n \equiv (\widetilde v_x^n, \widetilde u_x^n)$ as (piecewise constant) functions in $L^2(\Omega)$ by
\begin{align}
  \label{eq:dx1}
  \widetilde v_x^n(x) &:= \frac{1}{2 \Delta x} \left( {v^n_{i+1}-v^n_{i-1}} +\frac{\Delta x}{\Delta t} (u^n_{i+1} + u^n_{i-1} - 2 u^n_i)  \right)  &\quad \forall x \in \Omega_i, \\
  \label{eq:dx2}
  \widetilde u_x^n(x) &:= \frac{1}{2 \Delta x}  \left(u_{i+1}^n-u_{i-1}^n +   \frac{\Delta x}{\Delta t} (v^n_{i+1} + v^n_{i-1} - 2 v^n_i)  \right) &\quad \forall x \in \Omega_i.
\end{align}
Note that this choice of approximating the derivatives resembles a Lax-Friedrichs numerical flux with unit viscosity.
Consequently, one can approximate the quantities $\iota_{1}$ and $\iota_2$ by
\begin{alignat}{4}
 \iota_{1, h}(x) &:= v_i^n + {\Delta t}                   \widetilde u_x^n    	&\quad \forall x \in \Omega_i\\
 \iota_{2, h}(x) &:= g_i^n + \frac{1}{\varepsilon} \widetilde v_x^n		&\quad \forall x \in \Omega_i, 
\end{alignat}
which yields the following approximation $\iota_h$ to $\iota$:
\begin{align}
  \iota_h(\varphi) :=  \int_{\Omega} \iota_{1, h} \varphi - \Delta t^2 (1-\varepsilon) \iota_{2, h} \varphi_x \dx. 
\end{align}
The equation 
\begin{align}
  \label{eq:op_poisson_cont_approx}
  a(\overline v^{n+1}, \varphi) = \iota_h(\varphi) &\quad \forall \varphi \in H_0^1(\Omega) 
\end{align}
(being an approximation to \eqref{eq:op_poisson}) is now approximated by Finite Elements. Therefore, we define 
 \begin{align}
  V_h := \{\varphi_h \in C^0(\Omega) |{\varphi_h}_{|\Omega_i} \text{is linear for all } i; \varphi_h(0) = \varphi_h(1) = 0 \}
 \end{align}
and seek a solution $v_h^{n+1} \in V_h$, such that 
\begin{align}
  \label{eq:finite_element}
    a(v_h^{n+1}, \varphi_h) = \iota_h(\varphi_h) &\quad \forall \varphi_h \in V_h.
\end{align}
Subsequently, which constitutes the second step, we compute $u_h^{n+1}$ by 
\begin{align}
\label{eq:update_uh}
u_h^{n+1} = u^n + \dt \left(\frac{1}{\varepsilon} \widetilde v_x^{n} + \frac{1-\varepsilon}{\varepsilon^2} \frac{d}{dx} {v_h^{n+1}} + g^n \right), 
 \end{align}
see \eqref{eq:u_update}.
Values $u_i^{n+1}$ and $v_i^{n+1}$ are now obtained by evaluating $u_h^{n+1}$ and $v_h^{n+1}$, respectively, at cell-midpoints.
\subsection{(Order of) Consistency and some stability considerations}\label{sec:consistency}
{In this section, we show that our method is consistent, and we determine its order of consistency. The main theorem of this section is the following:

\begin{theorem}\label{thm:consistency}
 Let $v_h^{n+1}$ be the approximate solution according to the algorithm in Sec. \ref{sec:full_discr} with exact initial data $w^n \equiv w(t^n)$. Under Ass. \ref{ass:v} and \ref{ass:orders} (see below), we have 
 \begin{align}
 \|v_h^{n+1} - v(t^{n+1})\|_{L^2} = O\left(\varepsilon^2 \Delta t^2 + \varepsilon^4 + \varepsilon^3 \Delta x + \frac{\varepsilon^6}{\Delta x^2} \right).
 \end{align}
\end{theorem}

We have decided to put this investigation into the more classical framework of standard $H_0^1$ spaces and norms (instead of using $\Vs$), because in this setting we can use classical Finite-Element spaces and do not have to rely on stabilized Finite-Elements such as SUPG. This, however, comes at the price of restricting $\varepsilon$ to $0 < \varepsilon \leq \varepsilon_0 < 1$ and $\Delta t \geq \varepsilon$. Nevertheless, as we are interested in the $\varepsilon \rightarrow 0$ limit for a moderate time-step $\Delta t$, this is not a severe restriction.

To prove consistency of our scheme, we have to bound the following error parts:
\begin{align}
 e_1 &:= \|v(t^{n+1}) - v^{n+1}\|_{L^2} \\
 e_2 &:= \|v^{n+1} - \overline v^{n+1}\|_{L^2} \\
 e_3 &:= \|\overline v^{n+1} - v_h^{n+1}\|_{L^2}.
\end{align}
The overall consistency error in $v$, $e := \|v(t^{n+1}) - v_h^{n+1}\|_{L^2}$, can then be bounded by the sum of the $e_i$. 
Let us remind the reader of the following definitions:
\begin{itemize}
 \item $v(t^{n+1})$ denotes the exact solution $v$ to \eqref{eq:conslaw} at time $t^{n+1}$.
 \item $v^{n+1}$ denotes the exact solution to the elliptic equation, see \eqref{eq:op_poisson}.
 \item $\overline v^{n+1}$ denotes the solution to the elliptic equation \eqref{eq:op_poisson} with right-hand side $\iota_h$ instead of $\iota$, see \eqref{eq:op_poisson_cont_approx}.
 \item $v_h^{n+1}$ denotes the Finite-Element solution to the elliptic equation, see \eqref{eq:finite_element}. 
\end{itemize}
A schematic overview is given in Fig. \ref{fig:flowchart}. 

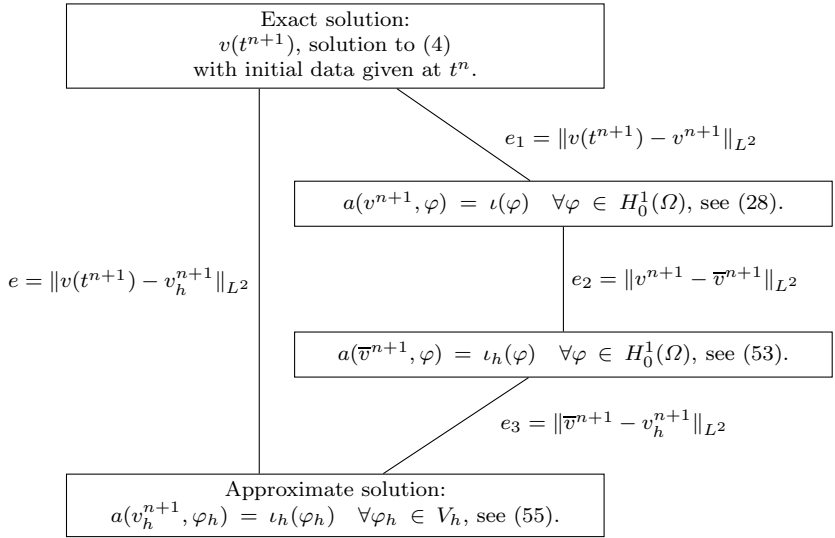
\begin{figure}[htb]
    \begin{center}
	\tikzstyle{decision} = [diamond, draw, 
	    text width=4.5em, text badly centered, node distance=2.2cm, inner sep=0pt]
	\tikzstyle{block_wide} = [rectangle, draw, 
	    text width=13em, text centered, minimum height=4em]
	\tikzstyle{block} = [rectangle, draw, 
	    text width=23em, text centered, minimum height=2em]    
	\tikzstyle{line} = [draw]
	\tikzstyle{cloud} = [draw, ellipse, minimum width = 35em, 
		node distance=3cm, text badly centered, minimum height=3em]
	\begin{tikzpicture}[scale=2,node distance = 2cm, auto]
	    \node [block, align=center] (start) 
	    {	
			Exact solution: \\ $v(t^{n+1})$, solution to \eqref{eq:conslaw} with initial data given at $t^n$. 
	    };
	    \node [block, below of=start, node distance=7em] (compute_primal) at (1.5,-.0)
	    {
			$a(v^{n+1}, \varphi) = \iota(\varphi) \quad \forall \varphi \in H_0^1(\Omega)$, see \eqref{eq:op_poisson}. 
	    };
 	    \node [block, below of=compute_primal] (compute_dual) 
 	    {
   			$a(\overline v^{n+1}, \varphi) = \iota_h(\varphi) \quad \forall \varphi \in H_0^1(\Omega)$, see \eqref{eq:op_poisson_cont_approx}.
 	    };
 	    \node [block, align=center, below of=start, node distance = 7em] (compute_more) at (0, -2.)
 	    {
   			Approximate solution:\\
			$a(v_h^{n+1}, \varphi_h) = \iota_h(\varphi_h) \quad \forall \varphi_h \in V_h$, see \eqref{eq:finite_element}.
 	    }; 	    
	    \path [line] (start) -- node[near end]{$e_1 = \|v(t^{n+1}) - v^{n+1}\|_{L^2}$} (compute_primal) ;
	    \path [line] (compute_primal) -- node[]{$e_2 = \|v^{n+1} - \overline v^{n+1}\|_{L^2}$} (compute_dual);
	    \path [line] (compute_dual) -- node[near start]{$e_3 = \|\overline v^{n+1} - v_h^{n+1}\|_{L^2}$} (compute_more);
	    \path [line] ($(compute_more.north)+(-.5,.0)$) --  node[]{$e =  \|v(t^{n+1}) - v_h^{n+1}\|_{L^2}$} ($(start.south) + (-.5,.0)$);
	\end{tikzpicture}
    \caption{Summary of steps in the consistency analysis. }
    \label{fig:flowchart}
    \end{center}
\end{figure} 

To obtain quantitative bounds on the consistency, we have to assume that both $u$ and $v$ are smooth. Furthermore, we make the following important assumption which is motivated by our investigations concerning the multiscale expansion, see \eqref{eq:conseq_multiscale}-\eqref{eq:conseq_multiscale2} and Rem. \ref{rem:order_vux}:
\begin{assumption}\label{ass:v}
 We assume that both $v$ and $u$ are sufficiently smooth. Furthermore, we assume that $v$ is given by
 \begin{align}
  v(x, t) = \varepsilon^2 v^{(2)}(x, t) + O(\varepsilon^3).
 \end{align}
 and that the spatial derivative of $u$ is given by
 \begin{align}
  u_x(x,t) = \varepsilon^2 u_x^{(2)}(x,t) + O(\varepsilon^3).
 \end{align}
 The big-$O$ notation has to be understood as in \eqref{eq:multiscale1_top}-\eqref{eq:multiscale2_top}.
\end{assumption}
\begin{remark}
 Without this assumption, it will not be possible to perform a consistency analysis for the small $\varepsilon$ limit, because there \emph{is} no limit function as $\varepsilon \rightarrow 0$. This is very similar to the observation in \cite{KlMa81} that the initial data has to be divergence free to allow for an incompressible limit. 
\end{remark}

It is well-known that, in order to get stable schemes, one needs to link both $\Delta t$ and $\Delta x$. In our example, this can be done in two ways, based on either the non-stiff flux $\widehat  f$ or the total flux $f$. Let us therefore make the following definition:
\begin{definition}
 The stiff and non-stiff $\cfl-$numbers $\widetilde \cfl$ and $\widehat \cfl$ are defined by 
 \begin{align}
  \widetilde \cfl := \frac{\Delta t}{\Delta x}  \lambda_{\max},  &\quad\quad
  \widehat \cfl  := \frac{\Delta t}{\Delta x}  \widehat   \lambda_{\max},
 \end{align}
 respectively, where $\lambda_{\max}$ is the maximum eigenvalue of the 'original' system \eqref{eq:conslaw}, and $\widehat \lambda_{\max}$ is the maximum eigenvalue of the non-stiff system corresponding to flux $\widehat f$. In the current case, $\lambda_{\max} = \varepsilon^{-1}$ and $\widehat \lambda_{\max} = 1$. 
\end{definition}

In our analysis, we rely on the non-stiff $\cfl-$number $\widehat \cfl$, so the $\cfl$ number that is independent on $\varepsilon$. It should, however, be less than unity, as the non-stiff part is treated explicitly. Let us state the following assumption:

\begin{assumption}\label{ass:orders}
 We assume that 
 \begin{align}
  \Delta t = \widehat \cfl \Delta x 
 \end{align}
 for a positive real-valued $\widehat \cfl < 1$ (which we usually choose to be $\widehat \cfl = 0.8$).
\end{assumption}

After these introductory statements, we start by bounding $e_1$. 
\begin{lemma}
The temporal discretization yields the following asymptotic error:
\begin{align}
 e_1&:= \|v(t^{n+1}) - v^{n+1}\|_{L^2} = O(\varepsilon^2 \dt^2).
\end{align}
\begin{proof}
 By checking the order of consistency of \eqref{eq:imex1}, one obtains ($t^n \leq \xi_1, \xi_2 \leq t^{n+1}$):
 \begin{align}
 & \frac{1}{\Delta t} \left(v(t^{n+1})-v(t^n) \right)  - \varepsilon u_x(t^n) - (1-\varepsilon) u_x(t^{n+1}) \\
 = & \frac{1}{\Delta t} \left(\Delta t \ v_t(t^n) + \frac{\Delta t^2}{2} v_{tt}(\xi_1)   \right) - \varepsilon u_x(t^n) - (1-\varepsilon) \left(u_x(t^n) + \Delta t \ u_{xt}(\xi_2) \right) \\
 = & v_t(t^n) + \frac{\Delta t}{2} v_{tt}(\xi_1) - u_x(t^n) - (1-\varepsilon) \Delta t \ u_{xt}(\xi_2) \\
 \stackrel{eq. \eqref{eq:p-system1}}{=}& \frac{\Delta t}{2} v_{tt}(\xi_1)- (1-\varepsilon) \Delta t \ u_{xt}(\xi_2) \\ \stackrel{Ass. \ref{ass:v}}{=} &O(\varepsilon^2\Delta t),
 \end{align}
 which yields indeed the desired order of accuracy. 
\end{proof}

\end{lemma}

Let us continue by bounding $e_2$. $w^n$ denotes the (assumed smooth) exact solution $w = (v, u)^T$ at time $t^n$. By $w_x^n$, we denote the exact derivative of $w$ at time $t^n$, and by $\widetilde w_x^n$, we denote the approximation of the derivative by numerical flux functions. 
We can state the following lemma:
\begin{lemma}
We consider approximations $\widetilde w_x^n$ to the derivatives $w_x^n$ as in \eqref{eq:dx1}-\eqref{eq:dx2}. Under Ass. \ref{ass:v} and \ref{ass:orders}, the following holds:
\begin{align}
 \label{eq:inequality_wx}
 \|w_x^n - \widetilde w_x^n\|_{L^2} = O(\varepsilon^2 \Delta x).
\end{align}
\end{lemma}
\begin{proof}
  A Taylor's series expansion yields that 
  \begin{align}
       & \frac{1}{2 \Delta x} \left( {v^n_{i+1}-v^n_{i-1}} +\frac{\Delta x}{\Delta t} (u^n_{i+1} + u^n_{i-1} - 2 u^n_i)  \right)  \\
   = \ & v^n_x(\overline x_i) + \frac{\Delta x}{4} \left(v^n_{xx}(\xi_1) - v^n_{xx}(\xi_2) \right) + u^n_{xx}(\xi_3) \frac{\Delta x^2}{2\Delta t}.
  \end{align}
  Note that $O(\frac{\Delta x}{\Delta t}) = O(1)$ (Ass. \ref{ass:orders}) and both $O(v_{xx}^n)$ and $O(u_{xx}^n)$ are $O(\varepsilon^2)$ (Ass. \ref{ass:v}).  Furthermore, 
  \begin{align}
   v^n_x(x) = v^n_x(\overline x_i) + v^n_{xx}(\xi_4) (x-\overline x_i) = v^n_x(\overline x_i) + O(\varepsilon^2 \Delta x) &\quad \forall x \in \Omega_i.
  \end{align} 
  Consequently, 
  \begin{align}
    \| v^n_x(\cdot) - \widetilde v^n_x(\cdot) \|_{L^2(\Omega)} &\leq \sum_{i=1}^{N_x} \left( \| v^n_x(\cdot) - v^n_x(\overline x_i) \|_{L^2(\Omega_i)} + \| v^n_x(\overline x_i) - \widetilde v^n_x(\cdot) \|_{L^2(\Omega_i)} \right) \\ 
    & = O(\varepsilon^2 \Delta x).
  \end{align}
  The proof for $u$ goes along the same lines. 
\end{proof}

Before considering the full approximation error, we have to turn to the variational equation \eqref{eq:op_poisson} again in the context of classical Sobolev-spaces. 
Following standard convention, we define the $H_0^1-$norm to be 
\begin{align}
 \|\varphi\|_{H_0^1} := \|\varphi_x\|_{L^2}, 
\end{align}
and remind the reader of Poincar\'e-Friedrich's inequality
\begin{align}
 \|\varphi\|_{L^2} \leq C_{PF} \|\varphi\|_{H_0^1}.
\end{align}
We start with the following theorem that guarantees that \eqref{eq:op_poisson} is, also for small $\varepsilon$, 'easy' to solve. 
\begin{theorem}\label{thm:konditionszahl}
 For a given $\varepsilon_0 < 1$, let $0 < \varepsilon \leq \varepsilon_0$, and $\Delta t \geq \varepsilon$. The equation \eqref{eq:op_poisson} is well-conditioned in $H_0^1$ independently of $\varepsilon$, which means that for two functionals $\iota, \widetilde \iota \in H_0^1(\Omega)'$, and the corresponding solutions $v$ and $\widetilde v$, one has the relation
 \begin{align}
  \label{eq:konditionszahl}
  \frac{\|v - \widetilde v\|_{H_0^1}}{\|v\|_{H_0^1}} \leq \frac{M}{\gamma} \frac{\|\iota - \widetilde \iota\|_{{H_0^1}'}}{\|\iota\|_{{H_0^1}'}}, 
 \end{align}
 and $\frac M\gamma$ can be bounded by a constant independent of $\varepsilon$.
\begin{proof}
It is easy to see that $a(\cdot, \cdot)$ fulfills, for $\varepsilon < 1$, an ellipticity condition on $H_0^1(\Omega)$ with ellipticity-constant $\gamma$, and it is a bounded bilinear form with stability constant $M$. Both $\gamma$ and $M$ can be explicitly given as
\begin{align}
\gamma = \frac{\dt^2 (1-\varepsilon)^2}{\varepsilon^2}, \quad
 M = \gamma + C_{PF}^2. 
\end{align}
The rest of the proof goes along the lines of Thm. \ref{thm:konditionszahl_cont}. 
Note that the quotient $\frac{M}{\gamma}$ is bounded for all $\varepsilon \leq \varepsilon_0 < 1$.
\end{proof}
\end{theorem}

\begin{remark}
 Thm. \ref{thm:konditionszahl} is an important result that can not be taken for granted. Standard codes will suffer from instabilities when small parameters, such as $\varepsilon$, occur. 
 Due to C\'ea's Lemma \cite{Ciarlet1978}, the $H^1-$error in a Finite-Element approximation of \eqref{eq:op_poisson} is bounded by $\frac M{\gamma}$ times the best-approximation error. 
\end{remark}

Let us return to our overall algorithm. Computing an approximate solution, we introduce two errors: One error from using a Finite-Element space instead of the whole Sobolev space, and one from considering $\iota_h$ instead of $\iota$. We start by computing the difference between the latter two: 

\begin{lemma}\label{la:iota_diff}
For a given $\varepsilon_0 < 1$, let $0 < \varepsilon \leq \varepsilon_0$. Furthermore, let $\iota$ and $\iota_h$ be defined as in Sec. \ref{sec:full_discr}. Its difference can be bounded in terms of $\Delta t$ and $\Delta x$ as 
\begin{align}
 \left\|\iota - \iota_h \right\|_{{H_0^1}'} = O\left(\varepsilon^2 \Delta t  \Delta x + {\Delta x \Delta t^2}{\varepsilon}\right).
\end{align}
\begin{proof}
From \eqref{eq:inequality_wx} and Ass. \ref{ass:v}, we can conclude that 
\begin{align}
 \left|\iota(\varphi) - \iota_h(\varphi) \right| 
  &= \left| \int_{\Omega} \left( \dt\left(u_x^n - \widetilde u_x^n\right) \varphi - \frac{\dt^2(1-\varepsilon)}{\varepsilon}\left(v_{x}^{n} - \widetilde v_x^n\right) \varphi_x \right)\dx \right| \\
  &\leq C \left(\Delta t + \frac{\Delta t^2}{\varepsilon}\right) \|w_x^n-\widetilde w_x^n \|_{L^2}  \|\varphi\|_{{H_0^1}} \\
  &= O\left(\varepsilon^2 \Delta t  \Delta x + {\Delta x \Delta t^2}{\varepsilon}\right) \|\varphi\|_{H_0^1}
\end{align}
for a constant $C \in \R$.
\end{proof}
\end{lemma}
The following lemma bounds the error that occurs when using only the approximate right-hand side $\iota_h$ instead of $\iota$:
\begin{lemma}
For a given $\varepsilon_0 < 1$, let $0 < \varepsilon \leq \varepsilon_0 < 1$. Furthermore, let $v^{n+1}$ and $\overline v^{n+1}$  denote the solutions to 
\begin{alignat}{2}
 \label{eq:op_poisson_approx}
 a(\overline v^{n+1}, \varphi) &= \iota_h(\varphi) &\quad& \forall \varphi \in H_0^1(\Omega), \\
 a( v^{n+1}, \varphi) &= \iota(\varphi) 	   &\quad& \forall \varphi \in H_0^1(\Omega). 
\end{alignat}
One can estimate the difference as 
\begin{align}
 \label{eq:error_pde}
 e_2 = \|\overline v^{n+1}-v^{n+1}\|_{L^2} = O\left(\varepsilon^4 \frac{\Delta x}{\Delta t} + \varepsilon^3 \Delta x\right).
\end{align}
\begin{proof}
The difference between $\overline v^{n+1}$ and $v^{n+1}$ can be computed by 
\begin{align}
\gamma \|\overline v^{n+1}-v^{n+1}\|^2_{{H_0^1}}  
  & \leq a(\overline v^{n+1}- v^{n+1}, \overline v^{n+1}-v^{n+1}) \\ 
  &= \iota_h(\overline v^{n+1} - v^{n+1}) - \iota(\overline v^{n+1} - v^{n+1}) \\
  &\leq \|\iota_h - \iota\|_{{H_0^1}'} \|\overline v^{n+1} - v^{n+1}\|_{{H_0^1}},
\end{align}
and, subsequently, 
\begin{align}
\|\overline v^{n+1}-v^{n+1}\|_{L^2} & \leq C_{PF} \|\overline v^{n+1}-v^{n+1}\|_{H_0^1} \\
  & \leq \frac{C_{PF}}{{\gamma}} \|\iota_h - \iota\|_{{H_0^1}'} = O\left(\varepsilon^4 \frac{\Delta x}{\Delta t} + \varepsilon^3 \Delta x\right)
\end{align}
because of La. \ref{la:iota_diff} and $\gamma^{-1} = O(\frac{\varepsilon^2}{\Delta t^2})$ for $\varepsilon, \Delta t \rightarrow 0$. 
\end{proof}
\end{lemma}
\begin{corollary}
 A simple consequence of the proof is that
 \begin{align}
  \label{eq:vqermvh1}
 \|\overline v^{n+1}-v^{n+1}\|_{{H_0^1}} = O\left(\varepsilon^4 \frac{\Delta x}{\Delta t} + \varepsilon^3 \Delta x\right).
 \end{align}
\end{corollary}

%

\begin{remark}
 Assumption \eqref{ass:orders} directly yields
 \begin{align}
  e_2 = O\left(\varepsilon^4 + \varepsilon^3 \Delta x\right).
 \end{align}
\end{remark}

Having bounded $e_2$, we continue by bounding $e_3$. 
\begin{lemma}
 Let $v_h^{n+1}$ be the Finite-Element solution to \eqref{eq:finite_element}, and let $\overline v^{n+1}$ be the solution to \eqref{eq:op_poisson_approx}. Then, 
 \begin{align}
  e_3 = \|\overline v^{n+1} - v_h^{n+1}\|_{L^2} = O\left(\frac{\varepsilon^6}{\Delta x^2} + \varepsilon^4 + \varepsilon^2 \Delta t^2\right).
 \end{align}
\begin{proof}
  We are using linear Finite-Elements on a symmetric problem, so one can use the Aubin-Nitsche trick (see, e.g., \cite{Ciarlet1978}). As it is crucial for our analysis that we get the correct dependency of the constant $\varepsilon$, we perform this 'trick' here explicitly. Let us define the dual solution $z$ and its Finite-Element approximation $z_h$ by
  \begin{align}
   a(z, \varphi) &= \int_{\Omega} \left(\overline v^{n+1}-v_h^{n+1} \right) \varphi \dx \quad \forall \varphi \in H_0^1(\Omega),  \\
   a(z_h, \varphi_h) &= \int_{\Omega} \left(\overline v^{n+1}-v_h^{n+1} \right) \varphi_h \dx \quad \forall \varphi_h \in V_h. 
  \end{align}
  One can conclude 
  \begin{align}
   \l2norm{\overline v^{n+1}-v_h^{n+1}}^2 
    &= a(z, \overline v^{n+1}-v_h^{n+1}) 
     = a(z-z_h, \overline v^{n+1}-v_h^{n+1}) \\
    &\leq M \|z-z_h\|_{H_0^1} \|\overline v^{n+1}-v_h^{n+1}\|_{H_0^1} \\
    &\leq M \Delta x^2 |z|_2 |\overline v^{n+1}|_2  \\
    &\leq C \frac{M \Delta x^2}{\gamma^2} \l2norm{\overline v^{n+1} -v_h^{n+1}} \| \iota_h \|_{{H_0^1}'}\\
    & = \l2norm{\overline v^{n+1} -v_h^{n+1}} \| \iota_h \|_{{H_0^1}'} O(\frac{\varepsilon^4}{\Delta x^2} + \varepsilon^2) \\
    & \leq \l2norm{\overline v^{n+1} -v_h^{n+1}} O(\varepsilon^2 + \Delta t^2) O(\frac{\varepsilon^4}{\Delta x^2} + \varepsilon^2) \\
    & \leq \l2norm{\overline v^{n+1} -v_h^{n+1}} O\left(\frac{\varepsilon^6}{\Delta x^2} + \varepsilon^4 + \varepsilon^2 \Delta t^2\right). 
  \end{align}
 $|\overline v|_2$ denotes the second Sobolev semi-norm. Considering an elliptic equation, it can be bounded by the right-hand side of the equation, if the ellipticity coefficient is unity. 
\end{proof}
 \end{lemma}
\begin{corollary}
 In a similar way, we can deduce that 
 \begin{align}
  \label{eq:vqermvhh1}
 \|\overline v^{n+1} - v_h^{n+1}\|_{H_0^1} = O\left(\frac{\varepsilon^6}{\Delta x^3} + \frac{\varepsilon^4}{\Delta x} + \varepsilon^2 \Delta t\right).
 \end{align}
\end{corollary}

We are now ready to prove the final theorem that assures that $v$ is approximated consistently. 
\begin{proof}[Of Thm. \ref{thm:consistency}] We can just collect previous results:

\begin{align}
  \|v(t^{n+1}) - v_h^{n+1}\|_{L^2} & \leq e_1 + e_2 + e_3 \\
				   &=  O(\varepsilon^2 \Delta t^2) + O\left(\varepsilon^4 + \varepsilon^3 \Delta x\right) + O\left(\frac{\varepsilon^6}{\Delta x^2} + \varepsilon^4 + \varepsilon^2 \Delta t^2\right) \\
				   &= O\left(\varepsilon^2 \Delta t^2 + \varepsilon^4 + \varepsilon^3 \Delta x + \frac{\varepsilon^6}{\Delta x^2}\right)
 \end{align}
\end{proof}

\begin{remark}
 Given that $\varepsilon \leq \Delta t$, one can see that $v_h^{n+1}$ is a consistent approximation to $v(t^{n+1})$, and $\|v(t^{n+1}) - v_h^{n+1}\|_{L^2} = O(\Delta t^4)$. 
\end{remark}

By now, we have shown that $v_h^{n+1}$ is a consistent approximation to $v(t^{n+1})$. It remains to show that also $u_h^{n+1}$ (see \eqref{eq:update_uh}) is a consistent approximation to $u(t^{n+1})$. 
\begin{theorem}\label{thm:consistency_u}
 Let $u_h^{n+1}$ be the approximate solution that is obtained using \eqref{eq:update_uh} with exact initial data $u^n \equiv u(t^n)$. Under Ass. \ref{ass:v} and \ref{ass:orders}, we have 
 \begin{align}
 \|u_h^{n+1} - u(t^{n+1})\|_{L^2} = O\left(\Delta t^2 + \frac{\varepsilon^4}{\Delta x^2} + \varepsilon^2\right).
 \end{align}
\begin{proof}
 We can directly compute, exploiting what we have already shown:
 \begin{align}
  \l2norm{u_h^{n+1}-u(t^{n+1})} 
  & \leq \l2norm{u^{n+1}-u(t^{n+1})} + \l2norm{u_h^{n+1} - u^{n+1}} \\
  & \leq O(\Delta t^2) + \l2norm{\frac{\Delta t}{\varepsilon} \left( v_x^n - \widetilde v_x^n \right)} + \l2norm{\Delta t \frac{1-\varepsilon}{\varepsilon^2} \left(v_h^{n+1} - v^{n+1}\right)_x} \\
  & \leq O(\Delta t^2) + \frac{\Delta t}{\varepsilon^2} \left(\h1norm{\overline v^{n+1} - v^{n+1}} + \h1norm{v_h^{n+1} - \overline v^{n+1}} \right) \\
  & \stackrel{eqs. \eqref{eq:vqermvh1}, \eqref{eq:vqermvhh1}}=  O(\Delta t^2) + O\left(\varepsilon^2 \Delta t + \varepsilon \Delta x^2 \right) + O\left(\frac{\varepsilon^4}{\Delta x^2} + \varepsilon^2 + \Delta t^2 \right) \\
  & = O\left(\Delta t^2 + \frac{\varepsilon^4}{\Delta x^2} + \varepsilon^2\right).
 \end{align}

\end{proof}
\end{theorem}

There are a few remarks in order:
\begin{remark}
  The solution of the elliptic equation gets more and more difficult with decreasing time-step $\Delta t$, as the elliptic coefficient vanishes in this case. So basically, the method will only perform well as long as $\varepsilon \leq \Delta t$ (i.e., for the $\cfl$ number of the whole system there holds $\cfl \leq \varepsilon$), which is a reasonable assumption. (Otherwise, one would use explicit methods instead.) However, choosing $\Delta t= O(\varepsilon^{\frac{1}{p}})$ for some $p \geq 1$, one can observe that 
  \begin{align}
   \l2norm{w_h^{n+1} - w(t^{n+1})} = O\left(\Delta t^2 \right) .
  \end{align}
  This directly shows that the method works also for the $\varepsilon = 0$ case. 
\end{remark}
}

\section{Numerical Results}\label{sec:numerics}
{

We compare our scheme with an Implicit-Euler scheme, and an Implicit/Explicit scheme. Implicit-Euler scheme discretizes
\begin{align}
 \frac{w^{n+1}-w^n}{\dt} + f(w^{n+1}) &= G^{n+1}
\end{align}
using a Lax-Friedrichs flux. The naive Implicit/Explicit scheme proceeds in two steps, discretizing
\begin{align}
 \frac{\widehat  w^{n}- w^n}{\dt} + \widehat f(w^n)_x &= G^n
\end{align}
explicitly, and 
\begin{align}
 \frac{w^{n+1}-\widehat w^n}{\dt} + \widetilde f(w^{n+1})_x &= 0
\end{align}
implicitly, again both steps with Lax-Friedrichs flux.

\subsection{Smooth test case}
As a first, simple test case, we consider a smooth solution on domain $\Omega = [0, 1]$, given by
\begin{align}
 v(x,t)  &= \varepsilon^2 t \sin(2\pi x) \\
 u(x,t)  &= \sin(20\pi t) - \frac{\varepsilon^2}{2\pi} {\cos(2\pi x)}.
\end{align}
For all methods, we use a (stiff) $\cfl$ number of $\widetilde \cfl = \frac{0.8}{\varepsilon}$. Note that this corresponds to a $\cfl$ number of $\widehat {\cfl} = 0.8$ with respect to the non-stiff flux $\widehat f$. If a method is able to cope with such a $\cfl$ number, it is called uniformly asymptotically stable. 
In Fig. \ref{fig:simple_test case3}, convergence of the $l^2-$error at time $T = 0.1$ versus number of cells ($N_x$) is shown for all three methods under consideration. 
Note that there is an erratic behavior in the beginning for all three methods. This is due to poor mesh resolution of the initial data. The asymptotic regime seems to start at $N_x = 16$. 
One can observe that all three methods are stable for this unusally large $\cfl$ number, as expected. Furthermore, asymptotically (in $N_x$), all methods converge with order one toward the true solution $(u, v)$, except for the Implicit Euler scheme for $\varepsilon = 10^{-8}$. We suspect that this is because the linear system of equations to be solved in each time-slab is extremely ill-conditioned. We use Matlab's in-house exact solver for linear systems of equations, which actually yields a corresponding warning. 
Furthermore, $\varepsilon^2 
= 10^{-16}$ is 
close to machine zero. Note that this does not happen to the 
Asymptotic Preserving 
scheme, as its condition number is bounded for $\varepsilon \rightarrow 0$. 

Furthermore, we can observe that the bounds given in Thm. \ref{thm:consistency} and Thm. \ref{thm:consistency_u} are too pessimistic. We suspect that the Finite-Element method performs in this case better than theoretically predicted. 

The really surprising outcome of this research is that the AP scheme performs so much better than Implicit Euler and the mixed Implicit / Explicit scheme: Its error is up to four orders of magnitude smaller than that of the other two schemes.  We can only suspect that 'traditional' Finite-Volume schemes do not take advantage of the smooth behaviour of the solution as much as the Finite-Element method does.
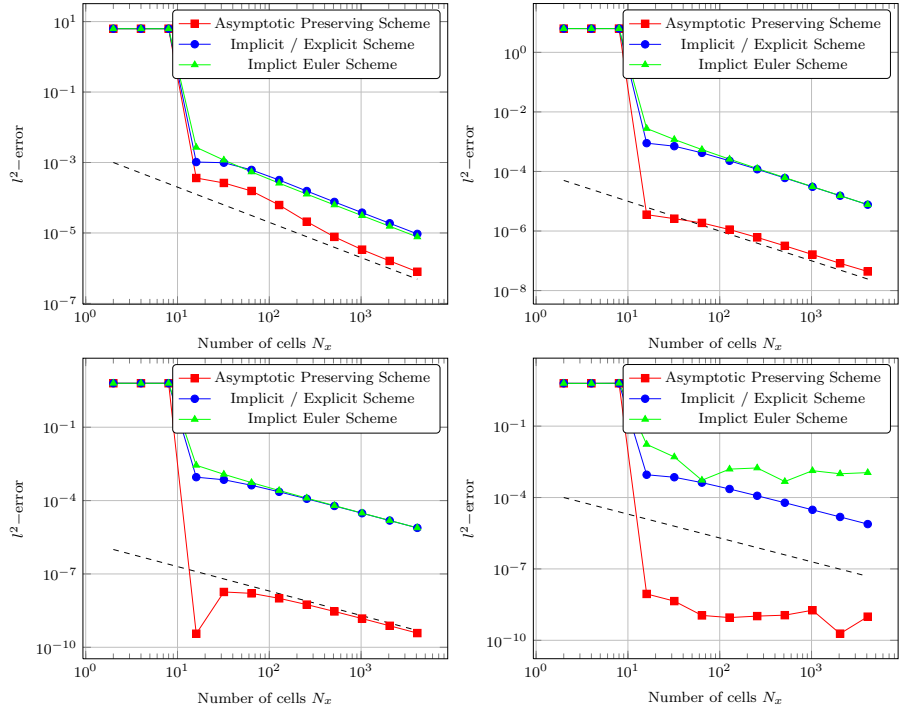
\begin{figure}[h]
	\begin{center}
	\begin{tikzpicture}[scale=0.70]
		\begin{loglogaxis}[xlabel={Number of cells $N_x$},ylabel={$l^2-$error},grid=major,legend style={at={(0.98,0.98)},anchor=north east,font=\footnotesize,rounded corners=2pt}]
		\addplot[color=red,   mark=square*]   table[x expr=(\thisrow{NE}), y expr =10^(\thisrow{eps1e1})] {AP_time.tex};
		\addplot[color=blue,  mark=*]         table[x expr=(\thisrow{NE}), y expr =10^(\thisrow{eps1e1})] {APO_time.tex};
		\addplot[color=green, mark=triangle*] table[x expr=(\thisrow{NE}), y expr =10^(\thisrow{eps1e1})] {IE_time.tex};
		\addplot[color=black, dashed] coordinates { (2,1e-3) (4096,0.488e-6) };
		\legend{Asymptotic Preserving Scheme, Implicit / Explicit Scheme, Implict Euler Scheme}
		\end{loglogaxis}
	\end{tikzpicture}
	\begin{tikzpicture}[scale=0.70]
		\begin{loglogaxis}[xlabel={Number of cells $N_x$},ylabel={$l^2-$error},grid=major,legend style={at={(0.98,0.98)},anchor=north east,font=\footnotesize,rounded corners=2pt}]
		\addplot[color=red,   mark=square*]   table[x expr=(\thisrow{NE}), y expr =10^(\thisrow{eps1e2})] {AP_time.tex};
		\addplot[color=blue,  mark=*]         table[x expr=(\thisrow{NE}), y expr =10^(\thisrow{eps1e2})] {APO_time.tex};
		\addplot[color=green, mark=triangle*] table[x expr=(\thisrow{NE}), y expr =10^(\thisrow{eps1e2})] {IE_time.tex};
		\addplot[color=black, dashed] coordinates { (2,5e-5) (4096,2.440e-8) };
		\legend{Asymptotic Preserving Scheme, Implicit / Explicit Scheme, Implict Euler Scheme}
		\end{loglogaxis}
	\end{tikzpicture}
	\begin{tikzpicture}[scale=0.70]
		\begin{loglogaxis}[xlabel={Number of cells $N_x$},ylabel={$l^2-$error},grid=major,legend style={at={(0.98,0.98)},anchor=north east,font=\footnotesize,rounded corners=2pt}]
		\addplot[color=red,   mark=square*]   table[x expr=(\thisrow{NE}), y expr =10^(\thisrow{eps1e4})] {AP_time.tex};
		\addplot[color=blue,  mark=*]         table[x expr=(\thisrow{NE}), y expr =10^(\thisrow{eps1e4})] {APO_time.tex};
		\addplot[color=green, mark=triangle*] table[x expr=(\thisrow{NE}), y expr =10^(\thisrow{eps1e4})] {IE_time.tex};
		\addplot[color=black, dashed] coordinates { (2,1e-6) (4096,0.488e-9) };
		\legend{Asymptotic Preserving Scheme, Implicit / Explicit Scheme, Implict Euler Scheme}
		\end{loglogaxis}
	\end{tikzpicture}
	\begin{tikzpicture}[scale=0.70]
		\begin{loglogaxis}[xlabel={Number of cells $N_x$},ylabel={$l^2-$error},grid=major,legend style={at={(0.98,0.98)},anchor=north east,font=\footnotesize,rounded corners=2pt}]
		\addplot[color=red,   mark=square*]   table[x expr=(\thisrow{NE}), y expr =10^(\thisrow{eps1e8})] {AP_time.tex};
		\addplot[color=blue,  mark=*]         table[x expr=(\thisrow{NE}), y expr =10^(\thisrow{eps1e8})] {APO_time.tex};
		\addplot[color=green, mark=triangle*] table[x expr=(\thisrow{NE}), y expr =10^(\thisrow{eps1e8})] {IE_time.tex};
		\addplot[color=black, dashed] coordinates { (2,1e-4) (4096,0.488e-7) };
		\legend{Asymptotic Preserving Scheme, Implicit / Explicit Scheme, Implict Euler Scheme}
		\end{loglogaxis}
	\end{tikzpicture}
	\caption{Convergence results for the smooth test case. In dependency on $\varepsilon$, $\cfl$ was set to  $\cfl = \frac{0.8}{\varepsilon}$. Left to right, top to bottom: $\varepsilon = \{10^{-1},10^{-2},10^{-4},10^{-8} \}$. Dashed line indicates first-order convergence.}\label{fig:simple_test case3}
	\end{center}
\end{figure}

%

\subsection{Testcase with a kink}
To assess whether the good performance of the asymptotic preserving method is due to the smoothness of the solution, we perform a numerical study on a test case with a kink, more precisely, we consider again domain $\Omega = [0, 1]$ and the solution 
\begin{align}
 v(x,t)  &= {\hphantom{1 + }}\varepsilon^2 t \begin{cases} x &\quad x < 0.5 \\ -x + 1 &\quad x \geq 0.5 \end{cases} \\
 u(x,t)  &= 1 + \varepsilon^2 \begin{cases} \frac{x^2}{2} &\quad x < 0.5 \\ -\frac{x^2}{2} + x - \frac14 &\quad x \geq 0.5 \end{cases}.
\end{align}
Again, we use a (stiff) $\cfl$ number of $\widetilde \cfl = \frac{0.8}{\varepsilon}$. In Fig. \ref{fig:simple_test case2}, convergence of the $l^2-$ norm at time $T = 0.1$ versus $N_x$ is plotted. 
One can observe that the schemes converge with order one up to $10^{-10}$, which is about machine zero (note that the error has to be scaled with $\varepsilon^2$), except for the $\varepsilon = 10^{-8}$, where Implicit Euler fails to converge for this large $\cfl$ number. 
For large values of $\varepsilon$, the schemes nearly perform equally well, while, for $\varepsilon= 10^{-4}$, the AP scheme really performs better by orders of magnitude. For $\varepsilon = 10^{-8}$, both the AP and Implicit / Explicit scheme perform about equally well. Nevertheless, as $\varepsilon^2 = 10^{-16}$ is close to machine zero, these results are not too reliable. 
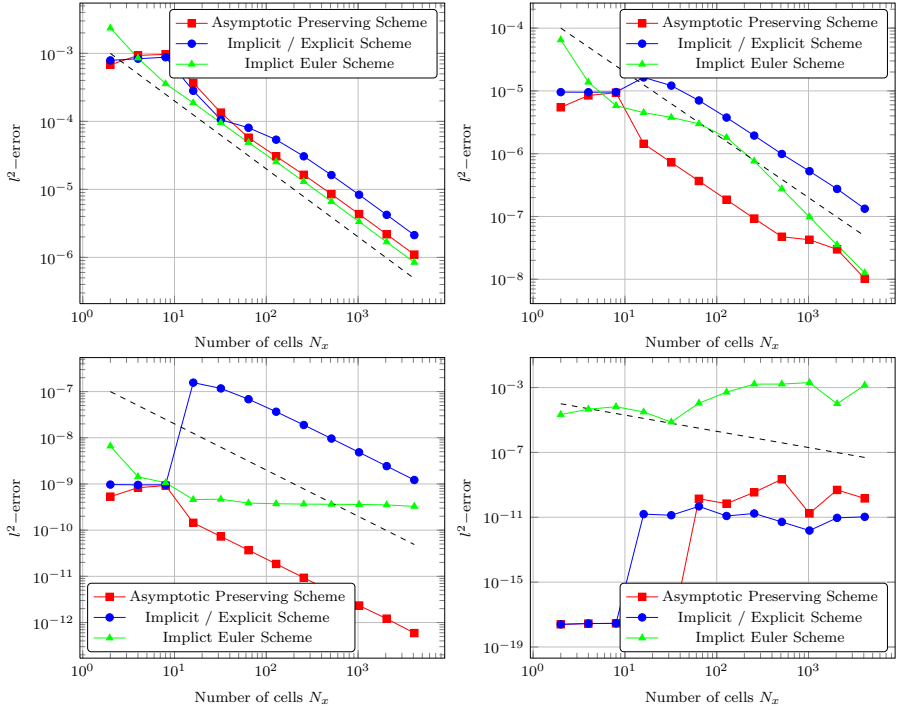
\begin{figure}[h]
	\begin{center}
	\begin{tikzpicture}[scale=0.70]
		\begin{loglogaxis}[xlabel={Number of cells $N_x$},ylabel={$l^2-$error},grid=major,legend style={at={(0.98,0.98)},anchor=north east,font=\footnotesize,rounded corners=2pt}]
		\addplot[color=red,   mark=square*]   table[x expr=(\thisrow{NE}), y expr =10^(\thisrow{eps1e1})] {AP_kink.tex};
		\addplot[color=blue,  mark=*]         table[x expr=(\thisrow{NE}), y expr =10^(\thisrow{eps1e1})] {APO_kink.tex};
		\addplot[color=green, mark=triangle*] table[x expr=(\thisrow{NE}), y expr =10^(\thisrow{eps1e1})] {IE_kink.tex};
		\addplot[color=black, dashed] coordinates { (2,1e-3) (4096,0.488e-6) };
		\legend{Asymptotic Preserving Scheme, Implicit / Explicit Scheme, Implict Euler Scheme}
		\end{loglogaxis}
	\end{tikzpicture}
	\begin{tikzpicture}[scale=0.70]
		\begin{loglogaxis}[xlabel={Number of cells $N_x$},ylabel={$l^2-$error},grid=major,legend style={at={(0.98,0.98)},anchor=north east,font=\footnotesize,rounded corners=2pt}]
		\addplot[color=red,   mark=square*]   table[x expr=(\thisrow{NE}), y expr =10^(\thisrow{eps1e2})] {AP_kink.tex};
		\addplot[color=blue,  mark=*]         table[x expr=(\thisrow{NE}), y expr =10^(\thisrow{eps1e2})] {APO_kink.tex};
		\addplot[color=green, mark=triangle*] table[x expr=(\thisrow{NE}), y expr =10^(\thisrow{eps1e2})] {IE_kink.tex};
		\addplot[color=black, dashed] coordinates { (2,1e-4) (4096,0.488e-7) };
		\legend{Asymptotic Preserving Scheme, Implicit / Explicit Scheme, Implict Euler Scheme}
		\end{loglogaxis}
	\end{tikzpicture}
	\begin{tikzpicture}[scale=0.70]
		\begin{loglogaxis}[xlabel={Number of cells $N_x$},ylabel={$l^2-$error},grid=major,legend style={at={(0.02,0.02)},anchor=south west,font=\footnotesize,rounded corners=2pt}]
		\addplot[color=red,   mark=square*]   table[x expr=(\thisrow{NE}), y expr =10^(\thisrow{eps1e4})] {AP_kink.tex};
		\addplot[color=blue,  mark=*]         table[x expr=(\thisrow{NE}), y expr =10^(\thisrow{eps1e4})] {APO_kink.tex};
		\addplot[color=green, mark=triangle*] table[x expr=(\thisrow{NE}), y expr =10^(\thisrow{eps1e4})] {IE_kink.tex};
		\addplot[color=black, dashed] coordinates { (2,1e-7) (4096,0.488e-10) };
		\legend{Asymptotic Preserving Scheme, Implicit / Explicit Scheme, Implict Euler Scheme}
		\end{loglogaxis}
	\end{tikzpicture}
	\begin{tikzpicture}[scale=0.70]
		\begin{loglogaxis}[xlabel={Number of cells $N_x$},ylabel={$l^2-$error},grid=major,legend style={at={(0.98,0.02)},anchor=south east,font=\footnotesize,rounded corners=2pt}]
		\addplot[color=red,   mark=square*]   table[x expr=(\thisrow{NE}), y expr =10^(\thisrow{eps1e8})] {AP_kink.tex};
		\addplot[color=blue,  mark=*]         table[x expr=(\thisrow{NE}), y expr =10^(\thisrow{eps1e8})] {APO_kink.tex};
		\addplot[color=green, mark=triangle*] table[x expr=(\thisrow{NE}), y expr =10^(\thisrow{eps1e8})] {IE_kink.tex};
		\addplot[color=black, dashed] coordinates { (2,1e-4) (4096,0.488e-7) };
		\legend{Asymptotic Preserving Scheme, Implicit / Explicit Scheme, Implict Euler Scheme}
		\end{loglogaxis}
	\end{tikzpicture}
	\caption{Convergence results for the test case with a kink. In dependency on $\varepsilon$, $\cfl$ was set to  $\cfl = \frac{0.8}{\varepsilon}$. Left to right, top to bottom: $\varepsilon = \{10^{-1},10^{-2},10^{-4},10^{-8} \}$. Dashed line indicates first-order convergence.}\label{fig:simple_test case2}
	\end{center}
\end{figure}

%
}

\section{Conclusions and Outlook}\label{sec:outlook}
We have compared the recently developed AP schemes versus more traditional Finite-Volume schemes for the $p-$system. It was demonstrated that the AP schemes outperform both Implicit Euler and an Implicit / Explicit scheme by orders of magnitude if there is a small parameter $\varepsilon$. 

We are interested in the use of high-order methods, also in the context of asymptotic preserving schemes. In particular, our interest lies in the use of Discontinuous Galerkin method \cite{DG2,ABCM}. Future work will therefore treat an asymptotic preserving discontinuous Galerkin scheme applied to \eqref{eq:p-system1}-\eqref{eq:p_system2} for various orders of consistency, and also compare performance of the AP schemes versus Diagonally-Implicit-Runge-Kutta (DIRK) \cite{JS13}. 
It is to be expected that the high order of consistency will reduce the effect that we could observe in this publication. Nevertheless, the use of AP schemes has some inherent advantages, such as the occurence of an elliptic equation, which is generally easier to solve than a hyperbolic problem. To conclude, we are positive that there will still be a benefit of using AP schemes.

\section*{Acknowledgement} I am thankful for fruitful discussions with Sebastian Noelle. 
Furthermore, I highly appreciate the careful reading and critical annotations from the anonymous reviewer which really helped me to improve the presentation in this paper.

\bibliographystyle{spmpsci}      

\begin{thebibliography}{}
\providecommand{\url}[1]{{#1}}
\providecommand{\urlprefix}{URL }
\expandafter\ifx\csname urlstyle\endcsname\relax
  \providecommand{\doi}[1]{DOI~\discretionary{}{}{}#1}\else
  \providecommand{\doi}{DOI~\discretionary{}{}{}\begingroup
  \urlstyle{rm}\Url}\fi

\end{thebibliography}


\begin{thebibliography}{10}
\providecommand{\url}[1]{{#1}}
\providecommand{\urlprefix}{URL }
\expandafter\ifx\csname urlstyle\endcsname\relax
  \providecommand{\doi}[1]{DOI~\discretionary{}{}{}#1}\else
  \providecommand{\doi}{DOI~\discretionary{}{}{}\begingroup
  \urlstyle{rm}\Url}\fi

\bibitem{ABCM}
Arnold, D.N., Brezzi, F., Cockburn, B., Marini, L.D.: Unified analysis of
  {Discontinuous Galerkin} methods for elliptic problems.
\newblock SIAM J. Numer. Anal. \textbf{39}, 1749--1779 (2002)

\bibitem{ArNo12}
Arun, K., Noelle, S.: An asymptotic preserving scheme for low froude number
  shallow flows.
\newblock IGPM Preprint 352  (2012)

\bibitem{ArNoLuMu12}
Arun, K., Noelle, S., Lukacova-Medvidova, M., Munz, C.D.: An asymptotic
  preserving all mach number scheme for the euler equations of gas dynamics.
\newblock IGPM Preprint 348  (2012)

\bibitem{BroHugh82}
Brooks, A.N., Hughes, T.J.R.: Streamline upwind/petrov-galerkin formulations
  for convection-dominated flows with particular emphasis on the incompressible
  navier-stokes equations.
\newblock Computer Methods in Applied Mechanics and Engineering \textbf{32},
  199--259 (1982)

\bibitem{Ciarlet1978}
Ciarlet, P.G.: The {Finite Element} Method for Elliptic Problems.
\newblock North-Holland, Amsterdam, New York, Oxford (1978)

\bibitem{DG2}
Cockburn, B., Shu, C.W.: {TVB Runge-Kutta} local projection {Discontinuous
  Galerkin} finite element method for {Conservation Laws} {II}: General
  framework.
\newblock Mathematics of Computation \textbf{52}, 411--435 (1988)

\bibitem{KC}
Cole, J.D., Kevorkian, J.: Perturbation Methods in Applied Mathematics.
\newblock Springer Berlin / Heidelberg / New York (1981)

\bibitem{CoDeKu12}
Cordier, F., Degond, P., Kumbaro, A.: An asymptotic-preserving all-speed scheme
  for the euler and navier-stokes equations.
\newblock Journal of Computational Physics \textbf{231}, 5685--5704 (2012)

\bibitem{DA}
Dafermos, C.M.: Hyperbolic {Conservation Laws} in {Continuum Physics}.
\newblock Springer Berlin / Heidelberg (2005)

\bibitem{DegLoNaNe12}
Degond, P., Lozinski, A., Narski, J., Negulescu, C.: An asymptotic-preserving
  method for highly anisotropic elliptic equations based on a micro-macro
  decomposition.
\newblock Journal of Computational Physics \textbf{231}, 2724--2740 (2012)

\bibitem{DeTa}
Degond, P., Tang, M.: All speed scheme for the low mach number limit of the
  isentropic euler equation.
\newblock Commun. Comput. Phys. \textbf{10}, 1--31 (2011)

\bibitem{GrRo}
Grossmann, C., Roos, H.G.: Numerical Treatment of Partial Differential
  Equations.
\newblock Springer Berlin / Heidelberg (2007)

\bibitem{JS13}
Jaust, A., Sch\"utz, J.: A temporally adaptive hybridized discontinuous
  galerkin method for instationary compressible flows.
\newblock Tech. rep., IGPM (2013).
\newblock Submitted to Computers \& Fluids on 07/23/2013

\bibitem{Jin99}
Jin, S.: Efficient asymptotic-preserving (ap) schemes for some multiscale
  kinetic equations.
\newblock SIAM J. Sci. Comput. \textbf{21}, 441--454 (1999)

\bibitem{Jin2012}
Jin, S.: Asymptotic preserving (ap) schemes for multiscale kinetic and
  hyperbolic equations: A review.
\newblock Riv. Mat. Univ. Parma \textbf{3}, 177--216 (2012)

\bibitem{Jin1998}
Jin, S., Pareschi, L., Toscani, G.: Diffusive relaxation schemes for multiscale
  discrete-velocity kinetic equations.
\newblock SIAM J. Numer. Anal \textbf{35}, 2405--2439 (1998)

\bibitem{KlMa81}
Klainerman, S., Majda, A.: Singular limits of quasilinear hyperbolic systems
  with large parameters and the incompressible limit of compressible fluids.
\newblock Comm. Pure Appl. Math. \textbf{34}, 481--524 (1981)

\bibitem{Kroener}
Kr\"oner, D.: Numerical Schemes for Conservation Laws.
\newblock Wiley Teubner (1997)

\end{thebibliography}

\end{document}